\documentclass{article}

\usepackage{PRIMEarxiv}

\usepackage[utf8]{inputenc} 
\usepackage[T1]{fontenc}    
\usepackage{hyperref}       
\usepackage{url}            
\usepackage{booktabs}       
\usepackage{amsfonts}       
\usepackage{nicefrac}       
\usepackage{microtype}      
\usepackage{lipsum}
\usepackage{fancyhdr}       
\usepackage{graphicx}       
\graphicspath{{media/}}     

\usepackage{euscript}
\usepackage{dsfont}
\usepackage{amsmath}
\usepackage{amsthm}

\newtheorem{thm}{Theorem}
\newtheorem{remark}{Remark}

\newtheorem{defin}{Definition}

\title{Submartingale Condition for Weak Convergence for Semi-Markov Processes
}

\author{
  Vitaliy Golomoziy \\
  Taras Shevchenko National University of Kyiv \\
  Kyiv\\
  \texttt{vitaliy.golomoziy@knu.ua} \\
}

\begin{document}
\maketitle

\begin{abstract}
In this paper, we consider a modified version of a well-known submartingale condition for the weak convergence of probability measures, adapted to the semi-Markov case. In this setting, it is convenient to work with an embedded Markov chain and the filtration generated by jump times. We demonstrate that a straightforward restatement of the classical result is not valid, and that an additional condition is required.
\end{abstract}


\section{Introduction}

The submartingale condition for weak convergence was introduced in the celebrated book \cite{strook}, Theorem 1.4.6, and is stated as follows:

\begin{thm}\label{strook_thm}[D. Strook, S. Varadhan]
  Let $\Omega = C([0,\infty); \mathbb{R}^d)$ be the space of continuous functions on $[0,\infty)$ with values in $\mathbb{R}^d$, and
  let $(\EuScript{M}_t)_{t\ge 0}$ be the corresponding canonical filtration.
  Let $\EuScript{P}$ be a family of probability measures on $(\Omega, (\EuScript{M}_t)_{t\ge 0})$, such that for any non-negative $f \in C^\infty_0(\mathbb{R}^d)$ there exists a
  constant $A_f$ such that, for all $P \in \EuScript{P}$, the stochastic process
  \[
    \left(f(x(t)) + A_f t, \EuScript{M}_t\right)_{t\ge 0}
  \]
  is a non-negative $P$-submartingale. Assume also that $A_f$ can be selected in such a way that it works for all translations of $f$ (i.e. if $g(x) = f(x-a)$, $a\in \mathbb{R}^d$, then $A_g = A_f$).
  Assume further that
  \begin{equation}\label{cond_1}
    \lim_{c\to \infty} \sup_{P\in \EuScript{P}} P(|x(0)| \ge c) = 0.
  \end{equation}
  Then the family $\EuScript{P}$ is weakly precompact (and thus tight).
\end{thm}

This condition is useful in many situations, especially in connection with diffusion processes and the associated martingale problem. However, in the theory
of semi-Markov processes we typically face a discrete-time martingale problem. 

For instance, it is used as a main tool for establishing weak convergence in \cite{korolyk} (see Preface, page vii).
Next, we give some basic definitions and recall the main facts about semi-Markov processes. We will use \cite{korolyk} as the main source, and we restrict our attention to processes with values in $\mathbb{R}^d$.

\begin{defin}
  Function $Q \colon \mathbb{R}^d \times \EuScript{B} \times [0,\infty) \to [0,1]$, where $\EuScript{B}$ is a Borel sigma-field in $\mathbb{R}^d$, called a
   \textit{semi-Markov kernel} on $(\mathbb{R}^d, \EuScript{B})$ if
   \begin{itemize}
   \item For every $x\in \mathbb{R}^d$ and $B\in \EuScript{B}$, function $Q(x, B, \cdot)$ is non-decreasing, right-continuous real function, such that $Q(x,B,0)=0$.
   \item For every $t\ge 0$, $Q(\cdot, \cdot, t)$ is a sub-Markov kernel on $(\mathbb{R}^d, \EuScript{B})$.
   \item $P(\cdot, \cdot) = Q(\cdot, \cdot, \infty)$ is a Markov kernel on $(\mathbb{R}^d, \EuScript{B})$.
   \end{itemize}
\end{defin}

\begin{defin}
  An $\mathbb{R}^d$-valued \emph{Markov renewal process} is a two-component, time-homogeneous Markov chain 
  $(x_n, \tau_n)$, $n \ge 0$ taking values in $\mathbb{R}^d \times [0,\infty)$,
  with transition probability defined by a semi-Markov kernel $Q$ as follows
  \[
    \mathbb{P}\!\left( x_{n+1} \in B, \ \tau_{n+1} - \tau_n \le t \ \middle|\ \EuScript{G}_n \right) = Q(x_n, B, t),
  \]
  for any integer $n \ge 0$, real $t \ge 0$, and Borel set $B \subset \mathbb{R}^d$. We assume that $\tau_0=0$. 
  Here $\left(\EuScript{G}_n\right)_{n \ge 0}$ is the natural filtration generated by $\{(x_n, \tau_n): n \ge 0\}$.
\end{defin}

In what follows we assume that semi-Markov kernel $Q(x, B, t)$ admits a representation
\begin{equation}\label{semi_mark_repr}
  Q(x, B, t) = P(x, B) F_x(t),
\end{equation}
where $P(x,B)$ is a Markov kernel, and $F_x$ is a distribution function for every $x$. 

Thus, $\tau_{n+1} - \tau_n$ is a holding time, and when \eqref{semi_mark_repr} holds, it has distribution $F_x$, where $x_n = x_{\tau_n} = x$. We denote its mean by
\begin{equation}\label{mx_def}
  m(x) = \int_0^\infty t \, F_x(dt) = \int_0^\infty \big(1 - F_x(t)\big)\, dt.
\end{equation}

\begin{defin}
  A \emph{semi-Markov process} associated with the Markov renewal process $(x_n, \tau_n)$, $n \ge 0$, is the stochastic process
  \[
    x(t) = x_{\nu(t)}, \quad t \ge 0,
  \]
  where
  \[
    \nu(t) = \sup\{n \ge 0 : \tau_n \le t\}, \quad t \ge 0,
  \]
  is the counting process of jumps.
\end{defin}

We also introduce the continuous version of $\tau_n$ by
\[
  \tau(s) = \tau_{\nu(s)},
\]
and the continuous filtration generated by the semi-Markov process by
\[
 \EuScript{F}_t = \sigma\big(x(s), \tau(s), 0 \le s \le t\big).
\]
  
\begin{defin}
  Let $q(x) = 1/m(x)$. We define the compensating operator $\mathbb{L}$ of the Markov renewal process $(x_n, \tau_n)$, $n \ge 0$ (or of the associated semi-Markov process $x(t)$, $t \ge 0$) by
  \[
  \mathbb{L}\varphi(x,t) = q(x)\!\left[\int_0^\infty F_x(ds) \int_{\mathbb{R}^d} P(x,dy)\,\varphi(y, t+s) - \varphi(x,t)\right],
  \]
  where $\varphi(x,t)$ is a function from an appropriate class of test functions.
  By convention, we set $\mathbb{L}\varphi(x, t) = 0$ when $q(x) = 0$.
\end{defin}

When $\varphi(x,t) = \varphi(x)$ does not depend on $t$, this expression reduces to
\[
  \mathbb{L}\varphi(x) = q(x)\left(\int_{\mathbb{R}^d} P(x,dy)\,\varphi(y) - \varphi(x)\right).
\]

Finally, from Proposition 1.4 in \cite{korolyk}, we know that the discrete-time process
\[
  Z^\varphi_n := \varphi(x_n, \tau_n) - \sum_{i=1}^n (\tau_i - \tau_{i-1})\,\mathbb{L}\varphi(x_{i-1}, \tau_{i-1}), 
  \quad n \ge 0,
\]
is a martingale with respect to the discrete filtration $\EuScript{G}_n = \sigma(x_k, \tau_k : 0 \le k \le n)$, $n \ge 0$.

In applications, we typically use this fact to establish a discrete-time version of the submartingale condition of Theorem~\ref{strook_thm}. Namely, we can establish the following:

\textbf{Condition D.}  
\textit{Let $\EuScript{U}$ be an index set, and let $x^u(t)$, $u \in \EuScript{U}$, be a family of semi-Markov processes (with associated Markov renewal processes $(x^u_n, \tau^u_n)$, $n \ge 0$).  
Assume that condition~\eqref{cond_1} holds, and that for any non-negative $\varphi \in C^\infty_0(\mathbb{R}^d)$ there exists a constant $A_\varphi \ge 0$, the same for all translations of $\varphi$, such that for every $u \in \EuScript{U}$ the \textbf{discrete-time} process
\[
  \bigl(\varphi(x^u_n) + A_\varphi \tau^u_n\bigr)_{n \ge 0}
\]
is a non-negative submartingale with respect to its natural filtration $(\EuScript{G}^u_n)_{n \ge 0}$.}

The question is whether \textbf{Condition D} is sufficient for tightness. The answer is no, as we show in Section~\ref{counterexample_section}, so an additional condition is required. We introduce and discuss this condition in Theorem~\ref{submart_tightness}. We will see that it is essential to ensure that there is no positive time interval during which no jumps occur for all semi-Markov processes in the family. In other words, the frequency of jumps must increase. We also show that in some important and typical applications this condition indeed holds.

This paper is organized as follows. Section~\ref{main_section} contains the main result, and Section~\ref{counterexample_section} presents a counterexample that justifies the additional condition in Theorem~\ref{submart_tightness}. Finally, Section~\ref{scale_section} is devoted to a special case—families of semi-Markov processes obtained by scaling a single given process in space and time—which are important examples in the theory of semi-Markov approximations.

\section{Main result}\label{main_section}
In this section we assume that $\Omega = D[0,\infty)$ is the Skorokhod space equipped with the standard Borel $\sigma$-field $\EuScript{B}(\Omega)$ (with respect to
the Skorokhod topology; see \cite{billingsley}, Chapter~16 for details).

Note that any semi-Markov process defined through a Markov renewal process $(x_n, \tau_n)_{n \ge 0}$ has trajectories in a subspace $\Omega^*$ consisting of functions that have at most one accumulation point of jump times. In other words, $\omega \in \Omega^*$ if and only if there exists a non-decreasing sequence
\[
0 \le t_0 \le t_1 \le \cdots \le t_n \le \cdots
\]
such that $\omega(t-) \neq \omega(t)$ if and only if $t \in \{t_n : n \ge 0\}$. For such $\omega \in \Omega^*$, we define a non-decreasing sequence of jump times by $\tau_n(\omega) = t_n$. For all other $\omega \in \Omega \setminus \Omega^*$, we set $\tau_n(\omega) = 0$. Since $\Omega^*$ is Borel-measurable, this construction yields a sequence of non-decreasing stopping times with respect to the natural coordinate filtration.

To simplify notation, we adopt the following convention. Let $(\xi_n)_{n \ge 0}$ be a sequence of random variables and let $\tau \ge 0$ be an integer-valued
random variable. We will write $\xi_\tau(\omega)$ for $\xi_{\tau(\omega)}(\omega)$.

\begin{thm}\label{submart_tightness}
  Let $\left(\zeta^u(t)\right)_{t \ge 0}$, $u \in \{0,1,2,\ldots\}$, be a sequence of semi-Markov processes with values in $\mathbb{R}^d$, and let $\EuScript{P} = \{P^u\}_{u \ge 0}$ be the corresponding family of distributions on the Skorokhod space $(\Omega, \EuScript{B}(\Omega))$.
Let $\{\tau_n(\omega) : n \ge 0\}$ be the non-decreasing sequence of jump times defined above, and define the corresponding value at jump time $\tau_n$ by
\[
X_n(\omega) = \omega(\tau_n(\omega)).
\]
  Assume the following conditions hold.
  \begin{itemize}
    \item[(i)] For every $T > 0$,
    \begin{equation}\label{sumbart_tightness_cond1}
      \lim_{a \to \infty} \limsup_u P^u\!\left(\left\{\omega \in \Omega : \sup_{t \in [0,T]} \bigl|\omega(t)\bigr| \ge a \right\}\right) = 0.
    \end{equation}
    \item[(ii)] Assume that for every non-negative $f \in C^\infty_0(\mathbb{R}^d)$ there exists a constant $A_f \ge 0$ such that the discrete-time process
    \[
      \bigl(f(X_n) + A_f \tau_n\bigr)_{n \ge 0}
    \]
    is a non-negative submartingale with respect to the filtration $\EuScript{G}_n = \sigma(X_k, \tau_k : 0 \le k \le n)$.  
    Assume also that the choice of $A_f$ can be made so that it works for all translates of $f$.
    \item[(iii)] For any $n$ and $t$, denote the next jump after $\tau_n + t$ by
    \begin{equation}\label{sumbart_tightness_hat_tau_def}
      \hat\tau_n(t) = \hat\tau_n(t; \omega) = \inf_{m > n} \{\tau_m(\omega) : \tau_m(\omega) > \tau_n(\omega) + t\},
    \end{equation}
    and define the conditional expectation of the time between $\tau_n + t$ and the next jump by
    \begin{equation}\label{submart_tightness_dn_def}
      d^u_x(t) = E^u\!\bigl[\hat\tau_n(t) - \tau_n - t \,\big|\, X_n = x\bigr], \quad x \in \mathbb{R}^d.
    \end{equation}
    Assume that
    \begin{equation}\label{sumbart_tightness_dens_of_jumps}
      \lim_{t \to 0} \limsup_{u \to \infty} \sup_{x \in \mathbb{R}^d} d^u_x(t) = 0.
    \end{equation}
  \end{itemize}
  Then the family $\EuScript{P}$ is tight in $D[0,\infty)$.
\end{thm}

\begin{remark}
  Note that $d^u_x(t)$ does not depend on $n$, since the pair $(X_n, \tau_n)$ forms a time-homogeneous, discrete-time Markov chain, and $\hat\tau_n(t) - \tau_n$ is
  independent of $\tau_k$, $k \le n$, and depends only on the state $X_n$.
\end{remark}
\begin{proof}

The proof follows the steps of the original proof of Theorem~1.4.6 from \cite{strook}.
First, for a function $y \in \Omega = D[0,\infty)$ and $\delta > 0$, we define
\[
  w^\prime_y(\delta; T) = \inf_{\{0=t_0 < t_1 < \ldots < t_n=T \}} 
  \max_{1 \le i \le n} \sup_{s,t \in [t_{i-1},t_i)} |y(t) - y(s)|,
\]
where the infimum is taken over all sets $\{t_0, \ldots, t_n\}$ such that  
$0 = t_0 < t_1 < \ldots < t_n = T$ ($n$ is arbitrary) and
$\min_{1 \le i < n} (t_i - t_{i-1}) > \delta$ (see \cite{billingsley}, p. 171 for details).

From \cite{billingsley}, Theorem~16.8, we know that the family of distributions $\EuScript{P}$ is tight if and only if condition~\eqref{sumbart_tightness_cond1}
holds, together with the following condition:
\begin{equation}\label{submart_tightness_cond2}
  \lim_{\delta \to 0} \limsup_u P^u\!\left(\left\{y \in \Omega : w^\prime_y(\delta; T) \ge \rho \right\}\right) = 0,
\end{equation}
for all $\rho > 0$ and $T > 0$.

Thus, our goal is to prove~\eqref{submart_tightness_cond2}. In what follows we assume that both $T$ and $\rho$ are fixed positive numbers.
Following the proof of Theorem~1.4.6 from \cite{strook}, we define for $\omega \in \Omega$:
\begin{equation}\label{submart_tightness_nudef}
  s_0 = 0, 
  \qquad
  s_n(\omega) = \inf\!\left\{ t \ge s_{n-1}(\omega) : \bigl|\omega(t) - \omega(s_{n-1})\bigr| \ge \rho/4 \right\}.
\end{equation}

Denote
\[
  N = N(\omega) = \min\{ n : s_{n+1}(\omega) > T\},
\]
and
\[
  \Delta_\rho(T; \omega) = \min\{ s_n(\omega) - s_{n-1}(\omega) : 1 \le n \le N(\omega)\}.
\]

Next, we make a crucial observation: each $P^u$ is concentrated on piecewise constant functions in $D[0,\infty)$, so the moments $s_n$ necessarily coincide
with some $\tau_m$. Hence we can define $\nu_n(\omega)$ as the integer such that
\[
  s_n(\omega) = \tau_{\nu_n}(\omega),
\]
where we use the convention $\tau_{\nu_n}(\omega) = \tau_{\nu_n(\omega)}(\omega)$ introduced at the beginning of this section.
Note that $\nu_n$ is a stopping time with respect to the filtration $(\EuScript{G}_m)_{m \ge 0}$.

We then observe that
\[
  P^u\!\left(\{y\in \Omega: w^\prime_y(\delta, T) \ge \rho \right\}) 
  \le P^u\!\left(\Delta_\rho(T) \le \delta\right),
\]
which follows directly from the definitions (the argument literally repeats the proof of Lemma~1.4.1 in \cite{strook}).

For each $\tilde\omega \in \Omega$, let $Q^u_{\tilde \omega}$ be a regular conditional probability,
\[
  Q^u_{\tilde \omega}(\cdot) = P^u(\cdot \mid \EuScript{G}_{\nu_n})(\tilde \omega),
\]
whose existence is guaranteed by Theorems~1.1.6 and~1.1.8 in \cite{strook}.
The corresponding expectation will be denoted by $E^u_{\tilde\omega}$, so that
\begin{equation}\label{submart_tightness_reg_code_exp_def}
  E^u_{\tilde\omega}[\xi] = \int_{\Omega} \xi(\omega)\, Q^u_{\tilde\omega}(d\omega) 
  = E^u\!\left[\xi \,\middle|\, \EuScript{G}_{\nu_n}\right](\tilde\omega).
\end{equation}

Let us now choose $f \in C^\infty_0(\mathbb{R})$ such that $f(0) = 1$, $f(x) = 0$ for $|x| \ge \rho/4$, and $0 \le f \le 1$. 
Define
\[
  \tilde f(x) = \tilde f(x, \tilde \omega) = f\!\left(x - X_{\nu_n}(\tilde \omega)\right),
\]
that is, a random translation of $f$ by $X_{\nu_n}(\tilde \omega)$.

Let $\gamma_{n,\delta}$ be the index of the first jump after $\tau_{\nu_n} + \delta$, so that 
$\hat \tau_{\nu_n}(\delta) = \tau_{\gamma_{n,\delta}}$, where $\hat \tau_n(t)$ was defined in~\eqref{sumbart_tightness_hat_tau_def}.
Note that $\gamma_{n,\delta}$ is also a stopping time with respect to the filtration $\left(\EuScript{G}_m\right)_{m \ge 0}$.

Fix an arbitrary $q \in \mathbb{Q}^n$, an $n$-dimensional vector of rational numbers. Put $f_q(x) = f(x-q)$ and note that $A_{f_q} = A_f$ by the
translation-invariance property of $A_f$ (see condition~(ii)).

Define
\[
  \kappa_{n,\delta} = \kappa_{n,\delta}(\omega) = \nu_{n+1}(\omega) \wedge \gamma_{n,\delta}(\omega).
\]
Note that $\kappa_{n,\delta}$ is a $\left(\EuScript{G}_m\right)_{m \ge 0}$-stopping time and that $\kappa_{n,\delta} \ge \nu_n$ $P^u$-a.s.\ for all $u$.
In what follows, we will write $\kappa$ and $\gamma$ to mean $\kappa_{n,\delta}$ and $\gamma_{n,\delta}$ respectively, as $n$ and $\delta$ are fixed.

By Proposition~IV.5.5 from \cite{neveu} and condition~(ii) of the theorem (namely, the submartingale property of the process 
$f(X_n) + A_f \tau_n$), we may conclude that there exists a null set
$F_q \in \EuScript{B}(\Omega)$ such that for all $\omega^\prime \notin F_q$,
\begin{align}\label{submart_tightness_disc_subm_ineq}
  E^u\!\left[ f_q\!\left(X_{\kappa}\right) + A_f \tau_{\kappa} \,\middle|\, \EuScript{G}_{\nu_n}\right](\omega^\prime) 
  \ge f_q\!\left(X_{\nu_n}(\omega^\prime)\right) + A_f \tau_{\nu_n}(\omega^\prime).
\end{align}

Let us define the null set
\[
  F = \bigcup_{q} F_q,
\]
and fix an arbitrary $\tilde \omega \notin F$. For this fixed $\tilde \omega$ and arbitrary $\varepsilon > 0$, we can find a rational vector
$\tilde q = \tilde q(\varepsilon, \tilde \omega) \in \mathbb{Q}^n$ such that
\[
  \sup_{x \in \mathbb{R}^d} \bigl| \tilde f(x) - f_{\tilde q}(x) \bigr|
  = \sup_{x \in \mathbb{R}^d} \bigl| f\!\left(x - X_{\nu_n}(\tilde \omega)\right) - f\!\left(x - \tilde q\right) \bigr| < \varepsilon.
\]
This is possible because $f$ has compact support and is therefore uniformly continuous.  
It is clear that $\tilde q(\cdot, \varepsilon)$ is $\EuScript{G}_{\nu_n}$-measurable as a function of $\omega$, since it is completely determined by $X_{\nu_n}(\omega)$.

Now, using~\eqref{submart_tightness_disc_subm_ineq}, we obtain
\begin{align*}\label{submart_tightness_disc_subm_ineq_2}
  E^u_{\tilde\omega}\!\left[ \tilde f(X_{\kappa}) + A_f \tau_{\kappa}\right] 
  &= E^u_{\tilde\omega}\!\left[ \bigl(\tilde f(X_\kappa) - f_{\tilde q}(X_\kappa)\bigr) + f_{\tilde q}(X_\kappa) + A_f \tau_{\kappa} \right] \\
  &\ge E^u_{\tilde\omega}\!\left[ f_{\tilde q}(X_\kappa) + A_f \tau_\kappa\right] - \varepsilon \\
  &\ge f_{\tilde q}\!\left(X_{\nu_n}(\tilde\omega)\right) + A_f \tau_{\nu_n}(\tilde\omega) - \varepsilon \\
  &\ge \tilde f\!\left(X_{\nu_n}(\tilde\omega)\right) + A_f \tau_{\nu_n}(\tilde\omega) - 2\varepsilon,
\end{align*}
where $E^u_{\tilde\omega}$ is defined in~\eqref{submart_tightness_reg_code_exp_def}.

Note that $\tilde \omega$ is fixed, so $\nu_n(\tilde \omega)$ is a non-random integer. Also observe that $\tilde f(X_{\nu_n}(\tilde\omega)) = f(0) = 1$.  
Finally, it is impossible that $\tau_{\nu_{n+1}} \in (\tau_{\nu_n} + \delta, \hat\tau_{\nu_n}(\delta))$, since $\tau_{\nu_{n+1}}$ is a jump time, and
$\hat\tau_{\nu_n}(\delta)$ is the first jump time after $\tau_{\nu_n} + \delta$.

Thus, we have a.s.
\[
  E^u_{\tilde\omega}\!\left[ \tilde f(X_{\kappa}) + A_f\bigl(\tau_\kappa - \tau_{\nu_n}(\tilde\omega)\bigr)\right] \ge 1 - 2\varepsilon.
\]
Hence,
\begin{align*}
  E^u_{\tilde\omega}\!\left[1 - \tilde f(X_{\kappa}) \right] 
  &\le E^u_{\tilde\omega}\!\left[A_f\bigl(\tau_\kappa - \tau_{\nu_n}(\tilde\omega)\bigr)\right] + 2\varepsilon \\
  &\le A_f\!\left(\delta + E^u_{\tilde\omega}\!\left[\hat\tau_{\nu_n}(\delta) - \tau_{\nu_n} - \delta\right]\right) + 2\varepsilon \\
  &=   A_f \delta + A_f d^u_{X_{\nu_n}(\tilde\omega)}(\delta) + 2\varepsilon \\
  &\le A_f\!\left(\delta + \sup_x d^u_x(\delta)\right) + 2\varepsilon,
\end{align*}
where $d^u_x(\delta)$ is defined in~\eqref{submart_tightness_dn_def}, and we used the fact that
\[
  \tau_\kappa - \tau_{\nu_n} = \tau_{\nu_{n+1}\wedge \gamma} - \tau_{\nu_n} 
  \le \tau_{\gamma} - \tau_{\nu_n} 
  = \hat\tau_{\nu_n}(\delta) - \tau_{\nu_n}.
\]
Note that $0 \le 1 - \tilde f \le 1$.

Define
\[
  B_{\tilde\omega} = \left\{ \omega \in \Omega : \tau_{\nu_{n+1}}(\omega) \le \tau_{\nu_n}(\tilde\omega) + \delta \right\} \subset \EuScript{G}_{\nu_{n+1}}.
\]
We know that for a regular conditional probability, $Q^u_{\tilde\omega}(C) = 1$ for $C \in \EuScript{B}(\Omega)$ if and only if $\tilde\omega \in C$
(see \cite{strook}, p.~16). Put
\[
  C_{\tilde\omega} = \{\omega \in \Omega : \nu_n(\omega) = \nu_n(\tilde\omega)\},
\]
and note that $\tilde\omega \in C_{\tilde\omega}$.

Using this fact and the definition of $f$, we obtain
\begin{align*}
  E^u_{\tilde\omega}\!\left[\tilde f(X_{\kappa}) \mathds{1}_{B_{\tilde\omega}}\right] 
  &= \int_{B_{\tilde\omega}} f\!\left(X_{\nu_{n+1}}(\omega) - X_{\nu_n}(\tilde \omega)\right)\, Q^u_{\tilde\omega}(d\omega) \\
  &= \int_{B_{\tilde\omega} \cap C_{\tilde\omega}} f\!\left(X_{\nu_{n+1}}(\omega) - X_{\nu_n}(\tilde\omega)\right)\, Q^u_{\tilde\omega}(d\omega) \\
  &= \int_{B_{\tilde\omega} \cap C_{\tilde\omega}} f\!\left(X_{\nu_{n+1}}(\omega) - X_{\nu_n}(\omega)\right)\, Q^u_{\tilde\omega}(d\omega) = 0,
\end{align*}
since
\[
  \bigl|X_{\nu_{n+1}}(\omega) - X_{\nu_n}(\omega)\bigr| \ge \rho/4.
\]

Thus, we obtain
\begin{align*}
  A_f\!\left(\delta + \sup_x d^u_x(\delta)\right) + 2\varepsilon
  &\ge E^u_{\tilde\omega}\!\left[(1 - \tilde f(X_{\kappa}))\bigl(\mathds{1}_{B_{\tilde\omega}} + \mathds{1}_{\Omega \setminus B_{\tilde\omega}}\bigr)\right] \\
  &=   Q^u_{\tilde\omega}(B_{\tilde\omega}) 
     + E^u_{\tilde\omega}\!\left[\bigl(1 - f(X_{\gamma})\bigr) \mathds{1}_{\Omega \setminus B_{\tilde\omega}}\right] \\
  &\ge Q^u_{\tilde\omega}(B_{\tilde\omega}),
\end{align*}
so we arrive at the inequality
\begin{equation}\label{submart_tightness_main_ineq}
  A_f\!\left(\delta + \sup_x d^u_x(\delta)\right) + 2\varepsilon 
  \ge Q^u_{\tilde\omega}(B_{\tilde\omega}) 
  = P^u\!\left(\tau_{\nu_{n+1}} \le \tau_{\nu_n}(\tilde\omega) + \delta \,\middle|\, \EuScript{G}_{\nu_n}\right)(\tilde\omega).
\end{equation}

Next, for every $k > 0$ we can write
\begin{align*}
  P^u\!\left(\{\omega : \Delta_\rho(T;\omega) \le \delta\}\right) 
  &\le P^u\!\left(\min_{1 \le i \le k} \tau_{\nu_i} - \tau_{\nu_{i-1}} \le \delta \right) + P^u(N > k) \\
  &\le \sum_{i=1}^k P^u\!\left(\tau_{\nu_i} - \tau_{\nu_{i-1}} \le \delta\right) + P^u(N > k) \\
  &\le \sum_{i=1}^k E^u\!\left[P^u\!\left(\tau_{\nu_i} - \tau_{\nu_{i-1}} \le \delta \,\middle|\, \EuScript{G}_{\nu_{i-1}}\right)\right] + P^u(N > k) \\
  &\le k\!\left[A_f\!\left(\delta + \sup_x d^u_x(\delta)\right) + 2\varepsilon\right] + P^u(N > k),
\end{align*}
where we used inequality~\eqref{submart_tightness_main_ineq} and the fact that $A_f$ can be chosen so that it works for all translations of $f$.
    
Next, we write for any $t_0 >0$:
\begin{align*}
  E^u\!\left[e^{-(\tau_{\nu_{i+1}}-\tau_{\nu_i})} \,\middle|\, \EuScript{G}_{\nu_i}\right] 
  &\le P^u\!\left(\tau_{\nu_{i+1}}-\tau_{\nu_i} \le t_0 \,\middle|\, \EuScript{G}_{\nu_i}\right) 
     + e^{-t_0} P^u\!\left(\tau_{\nu_{i+1}}-\tau_{\nu_i} > t_0 \,\middle|\, \EuScript{G}_{\nu_i}\right) \\
  &\le e^{-t_0} + (1-e^{-t_0})\, P^u\!\left(\tau_{\nu_{i+1}}-\tau_{\nu_i} \le t_0 \,\middle|\, \EuScript{G}_{\nu_i}\right) \\
  &\le e^{-t_0} + (1-e^{-t_0})\!\left[ A_f \left(t_0 + \sup_x d^u_{x}(t_0)\right) + 2\varepsilon \right],
\end{align*}
where we used~\eqref{submart_tightness_main_ineq}.  

From condition~(iii) of the theorem we know that
\[
  \limsup_{u \to \infty} \sup_x d^u_x(t_0) \to 0, \qquad t_0 \to 0,
\]
so we can choose $u_0$, $t_0$ and $\varepsilon$ such that
\[
  \lambda := e^{-t_0} + (1-e^{-t_0})\!\left[ A_f \bigl(t_0 + \sup_x d^u_{x}(t_0)\bigr) + 2\varepsilon \right] < 1,
\]
for all $u>u_0$.

Next, by Lemma~1.4.5 from \cite{strook}, we conclude that for $u>u_0$,
\[
  P^u(N>k) \le e^T \lambda^k, \qquad k \ge 0.
\]

Thus, in order to prove the theorem, it remains to show that condition~\eqref{submart_tightness_cond2}, namely
\[
  \lim_{\delta \to 0} \limsup_{u \to \infty} P^u\!\left(\{\omega : \Delta^u_\rho(T;\omega) \le \delta\}\right) = 0,
\]
holds.  

Indeed, we have
\begin{align*}
  \lim_{\delta \to 0} \limsup_{u \to \infty} &P^u \!\left(\{\omega : \Delta^u_\rho(T;\omega) < \delta\}\right) \\
  &\le \lim_{\delta \to 0} \limsup_{u \to \infty} 
      \Bigl(k\bigl[A_f (\delta + \sup_x d^u_{x}(\delta)) + 2\varepsilon \bigr] + P^u(N>k)\Bigr) \\
  &\le 2k\varepsilon + e^T \lambda^k,
\end{align*}
where in the last step we used condition~(iii).  

Now, we set $\varepsilon\to 0$ and then $k\to \infty$.
Thus, condition~\eqref{submart_tightness_cond2} holds, and the theorem is proven.

\end{proof}
\begin{remark}
  Condition~(iii) can be replaced with a somewhat opposite assumption: \\[0.5ex]
  \textbf{Condition (iv).} Assume that there exists a real number $a > 0$ such that for every $x \in \mathbb{R}^d$,
  \begin{equation}\label{condition_iv}
    \limsup_u P^u\!\left(\tau_{n+1} - \tau_n < a \,\middle|\, X_n = x \right) = 0.
  \end{equation}
  This condition means that asymptotically there are no jumps on the interval $(\tau_n, \tau_n + a)$. 
  The proof of tightness is trivial in this case, since 
  equation~\eqref{submart_tightness_cond2} takes the form (here we follow the notation from the proof)
  \begin{align*}
    \lim_{\delta \to 0} \limsup_{u \to \infty} 
      &P^u\!\left(\{\omega : \Delta^u_\rho(T;\omega) \le \delta\}\right)  \\
    &\le \lim_{\delta \to 0} \limsup_{u \to \infty} \left(\sum_{i=1}^k E^u\!\left[P^u\!\left(\tau_{\nu_i} - \tau_{\nu_{i-1}} \le \delta \,\middle|\, \EuScript{G}_{\nu_{i-1}}\right)\right] + P^u(N>k)\right) \\
    &= \limsup_{u\to \infty} P^u(N>k) \;\longrightarrow\; 0, \qquad k \to \infty.
  \end{align*}
\end{remark}

\section{Discussion and examples}\label{counterexample_section}
In the previous section we saw that each of Conditions~(iii) or~(iv) guarantees tightness (assuming the submartingale conditions~(i) and~(ii)).
At the same time, it is clear that neither Condition~(iii) nor Condition~(iv) is sufficient on its own.  
The following example shows two things:
\begin{itemize}
  \item These conditions cannot be omitted entirely;
  \item Tightness fails in situations where Conditions~(iii) and~(iv) are mixed in some sense.
\end{itemize}

Consider the sequence of semi-Markov processes $\zeta_n(t)$, $n \in \{1,2,3,\ldots\}$.  
Fix $n$. Let $\zeta_n(t)$ start at $0$ and make a jump to $1$ at one of two times, $1/n$ or $1$, each with probability $1/2$. 
After that, $\zeta_n(t)$ remains equal to $1$ forever.

To show that tightness does not hold, we use the same criterion from \cite{billingsley}, Theorem~16.8, that was applied in the proof of
Theorem~\ref{submart_tightness}. Theorem~16.8 states that a family of probability measures is tight if and only if both conditions~\eqref{sumbart_tightness_cond1} and~\eqref{submart_tightness_cond2} hold.  
We will show that condition~\eqref{submart_tightness_cond2} fails in this case.  

Recall the definition of $w^\prime$ from the proof of Theorem~\ref{submart_tightness}.  
For $y \in D[0,\infty)$ and $\delta > 0$ define
\[
  w^\prime_y(\delta; T) = \inf_{\{0=t_0 < t_1 < \ldots < t_n = T\}} 
    \max_{1 \le i \le n} \sup_{s,t \in [t_{i-1},t_i)} |y(t) - y(s)|,
\]
where the infimum is taken over all partitions $\{t_0, \ldots, t_n\}$ such that $0 = t_0 \le t_1 < \ldots < t_n = T$ ($n$ arbitrary) and
$\min_{1 \le i < n} (t_i - t_{i-1}) > \delta$.

Our goal is to show that
\[
  \lim_{\delta \to 0} \limsup_{n \to \infty} 
  P^n\!\left(\left\{ y \in \Omega : w^\prime_y(\delta; T) > \rho \right\}\right) > 0.
\]

Indeed, fix arbitrary $\rho \in (0,1)$ and $\delta > 0$.  
We can choose $n$ large enough so that $1/n < \delta$, which gives
\[
  \limsup_{n \to \infty} 
  P^n\!\left(\left\{ y \in \Omega : w^\prime_y(\delta; T) > \rho \right\}\right) = \tfrac{1}{2}.
\]

On the other hand, the submartingale property does hold. It suffices to check it at the first jump:
\[
  f(0) \le E^n[f(1) + A_f \tau_1] = f(1) + \tfrac{1}{2}\bigl(A_f/n + A_f\bigr),
\]
which is true if $A_f = 4 \sup |f|$.

Next, we demonstrate how Condition~(iii) can be verified in a situation where a family of semi-Markov processes is obtained via a time-scale and space-scale
transformation of a single process, which is typical in diffusion and averaging approximations.

\section{Space-time scaled semi-Markov processes}\label{scale_section}

Assume the semi-Markov process $x(t)$ is given in the sense of Definition~2, with associated Markov renewal process $(x_n, \tau_n)_{n \ge 0}$ and an associated family of holding-time distributions $\{F_x : x \in \mathbb{R}^d\}$.  
Let $\theta_n = \tau_{n+1} - \tau_n$. Note that $\theta_n$ depends only on $x_n$ and does not depend on $n$ or $\tau_n$.

We then generate a family of processes $(\zeta^n(t))_{n \ge 1}$ defined by
\[
  \zeta^n(t) = \frac{x(a_n t)}{b_n}, \qquad t \ge 0,
\]
where $\{a_n\}_{n \ge 1}$ and $\{b_n\}_{n \ge 1}$ are two increasing sequences of positive numbers such that $a_n \wedge b_n \to \infty$ as $n \to \infty$.  
Typical examples are $b_n = n$ and $a_n = n^2$ for diffusion schemes, or $a_n = n$ for averaging schemes (see \cite{korolyk} for details).
We denote by $(X^n_j, \tau^n_j)$ a Markov renewal process that is associated with the semi-Markov process $(\zeta^n(t))_{t\ge 0}$.

It follows from the construction that all processes $(x(t))_{t \ge 0}$, $(x_j, \tau_j)_{j \ge 0}$, $(\zeta^n(t))_{t \ge 0}$, and $(X^n_j, \tau^n_j)_{j \ge 0}$ are
defined on the same probability space.  
The next theorem gives a condition that allows one to verify condition~(iii).

\begin{thm}\label{time_scale_thm}
  Let $(\zeta^n)_{n \ge 1}$ be the family of space-time-scaled semi-Markov processes defined above.  
  Assume the following condition holds:
  \begin{equation}\label{time_scale_main_cond}
    \lim_{t \to 0} \limsup_{n \to \infty} \frac{1}{a_n} 
      \sup_{x} \int_{a_n t}^\infty \frac{\bar{F}_x(r)}{\bar{F}_x(a_n t)} \, dr = 0,
  \end{equation}
  where $\bar{F}_x = 1 - F_x$ is the tail distribution function associated with $F_x$.  
  Then condition~(iii) of Theorem~\ref{submart_tightness} holds.
\end{thm}
\begin{proof}
Let $P_{x,s}(\mathrm{d}y,\mathrm{d}t)$ be the transition probability of the Markov chain $(x_n,\tau_n)_{n\ge 0}$, and let
$P^{j}_{x,s}(\mathrm{d}y,\mathrm{d}t)$ denote the corresponding $j$-step transition probability.

Consider the measurable space $(\Omega,\EuScript{F})$, where
\[
\Omega = \bigl(\mathbb{R}^d \times [0,\infty)\bigr)^{\infty}
\]
is a countable product space and $\EuScript{F}$ is the corresponding cylinder $\sigma$-field.
Let
\[
\omega = \bigl((\omega_{00},\omega_{01}),(\omega_{10},\omega_{11}),\ldots\bigr) \in \Omega.
\]
Define the random variables
\[
X_j(\omega)=\omega_{j0}, 
\qquad 
\tau_j(\omega)=\omega_{j1},
\qquad 
X_j^n=\frac{X_j}{b_n},
\qquad 
\tau_j^n=\frac{\tau_j}{a_n}.
\]

For every pair $(x,s)\in\mathbb{R}^d\times[0,\infty)$ we may define a probability measure $\mathbb{P}_{x,s}$ on $(\Omega,\EuScript{F})$ such that the finite-dimensional
distributions of the sequence $(X_j,\tau_j)_{j\ge 0}$ coincide with those of the original sequence $(x_j,\tau_j)_{j\ge 0}$ starting at $(x,s)$.
The same is true for the scaled sequence $(X_j^n,\tau_j^n)_{j\ge 0}$.
In fact, not only the finite-dimensional distributions coincide, but also the distributions of the entire sequences as random elements of
$(\Omega,\EuScript{F})$; however, for our purposes it suffices to work with finite-dimensional distributions.

  Condition (iii) may then be restated as
\begin{equation}\label{time_scale_ineq0}
 \limsup_n \sup_{x,s} \mathbb{E}_{x,s}\left[\hat\tau^n_0(t) - s - t\right] \to 0, 
  \quad t \to 0.
\end{equation}
Moreover, it is enough to prove the convergence for $s=0$
\begin{equation}\label{time_scale_ineq1}
 \limsup_n \sup_{x} \mathbb{E}_{x,0}\left[\hat\tau^n_0(t) - t\right] \to 0, 
  \quad t \to 0,
\end{equation}
where $\hat\tau^n_0(t) = \inf_{m>0} \{\tau^n_m : \tau^n_m > \tau^n_0 + t\}$.

In what follows we will simplify the notation by denoting
\[ \mathbb{E}_x = \mathbb{E}_{x,0},\ d^n_x(t) = \mathbb{E}_{x}\left[\hat\tau^n_0(t) - t\right].\]

We can write
\begin{align*}
  d^n_x(t) 
  &= \mathbb{E}_x\!\left[\hat\tau^n_0(t) - t\right] 
   = \sum_{j=0}^\infty \mathbb{E}_x\!\left[\hat\tau^n_0(t) - t;\, \tau^n_j \le t < \tau^n_{j+1}\right] \\
  &= \sum_{j=0}^\infty \mathbb{E}_x\!\left[\tau^n_{j+1} - t;\, \tau^n_j \le t < \tau^n_{j+1}\right] \\
  &= \sum_{j=0}^\infty \mathbb{E}_x\!\left[
      \mathbb{E}_x\!\left[\tau^n_{j+1} - t;\, \tau^n_j \le t < \tau^n_{j+1} \,\middle|\, X^n_j, \tau^n_j \right]
    \right] \\
  &= \sum_{j=0}^\infty \mathbb{E}_x\!\left[
      \mathbb{E}_x\!\left[\theta^n_j - (t-\tau^n_j);\, \theta^n_j > (t-\tau^n_j) \,\middle|\, X^n_j, \tau^n_j \right]
      \mathds{1}_{\{\tau^n_j \le t\}}
    \right],
\end{align*}
where we used the relation $\theta^n_j = \tau^n_{j+1} - \tau^n_j$. Thus, by the Markov property,
\begin{align*}
  d^n_x(t) 
  &= \sum_{j=0}^\infty \mathbb{E}_x\!\left[
       \mathbb{E}_{X^n_j, \tau^n_j}\!\left[\theta^n_0 - (t-\tau^n_j);\, \theta^n_0 > t - \tau^n_j \right]
       \mathds{1}_{\{\tau^n_j \le t\}}
     \right] \\
  &= \sum_{j=0}^\infty \mathbb{E}_x\!\left[
       \mathbb{E}_{X^n_j, \tau^n_j}\!\left[\tfrac{\theta_0}{a_n} - \!\left(t - \tfrac{\tau_j}{a_n}\right);\,
          \tfrac{\theta_0}{a_n} > \!\left(t - \tfrac{\tau_j}{a_n}\right)\right]
       \mathds{1}_{\{\tau_j \le a_n t\}}
     \right] \\
  &= \frac{1}{a_n} \sum_{j=0}^\infty 
     \mathbb{E}_x\!\left[
       \mathbb{E}_{X^n_j, \tau^n_j}\!\left[\theta_0 - (a_n t-\tau_j);\,
         \theta_0 > a_n t - \tau_j\right]
       \mathds{1}_{\{\tau_j \le a_n t\}}
     \right] \\
  &= \frac{1}{a_n} \sum_{j=0}^\infty 
     \int_{\mathbb{R}^d}\int_0^{a_n t} 
       \mathbb{E}_{y/b_n}\!\left[\theta_0 - (a_n t-s);\,
         \theta_0 > a_n t - s\right] 
       P^j_{x,0}(dy, ds),
\end{align*}
where we used the equalities $X^n_j = x_j/b_n$ and $\tau^n_j = \tau_j/a_n$, together with the fact that $\theta_0$ does not depend on $\tau_0$.  
Hence we may omit the second subscript and write $E_{y/b_n}$ instead of $E_{y/b_n, s/a_n}$.

Finally, since $\theta_0$ is a nonnegative random variable, for every $z\in \mathbb{R}^d$
\begin{align*}
  \mathbb{E}_{z}\!&\left[\theta_0 - (a_n t-s);\ \theta_0 > a_n t - s\right] \\
  &= \left(\int_0^\infty 
       \mathbb{P}_z\!\left(\theta_0 > r+ (a_n t-s)\,\middle|\, \theta_0 > a_n t - s\right) dr\right) 
       \mathbb{P}_z(\theta_0 > a_n t-s) \\
  &= \left(\int_0^\infty 
        \frac{1 - F_z(r+ (a_n t-s))}{1-F_{z}(a_n t - s)}\, dr\right) 
       \bigl(1-F_z(a_n t-s)\bigr) \\
  &= \left(\int_0^\infty 
        \frac{\bar{F}_z(r+ (a_n t-s))}{\bar{F}_z(a_n t - s)}\, dr\right) 
       \bar{F}_z(a_n t-s).
\end{align*}

Thus, we can continue
\begin{align*}
  d^n_x(t)  
  &= \frac{1}{a_n} \sum_{j=0}^\infty \int_{\mathbb{R}^d}\int_0^{a_n t} 
       \int_0^\infty 
         \frac{\bar{F}_{y/b_n}(r+a_n t-s)}{\bar{F}_{y/b_n}(a_n t - s)}\,dr \,
         \bar{F}_{y/b_n}(a_n t-s)\, P^j_{x,0}(dy, ds) \\
  &= \frac{1}{a_n} \sum_{j=0}^\infty \int_{\mathbb{R}^d}\int_0^{a_n t} 
       \int_{a_n t-s}^\infty 
         \frac{\bar{F}_{y/b_n}(r)}{\bar{F}_{y/b_n}(a_n t - s)}\,dr \,
         \bar{F}_{y/b_n}(a_n t-s)\, P^j_{x,0}(dy, ds) \\
  &\le \left(\frac{1}{a_n}\sup_{y}\int_{a_n t}^\infty 
           \frac{\bar{F}_y(r)}{\bar{F}_y(a_n t)}\,dr\right) 
       \sum_{j=0}^\infty \int_{\mathbb{R}^d}\int_0^{a_n t}  
         \bar{F}_{y/b_n}(a_n t-s)\, P^j_{x,0}(dy, ds) \\
  &= \left(\frac{1}{a_n}\sup_{y}\int_{a_n t}^\infty 
           \frac{\bar{F}_y(r)}{\bar{F}_y(a_n t)}\,dr\right) 
       \sum_{j=0}^\infty 
         \mathbb{E}_x\!\left[\mathbb{P}_{X_j,\tau_j}(\theta_0 > a_n t-\tau_j)\,
         \mathds{1}_{\{\tau_j \le a_n t\}}\right] \\
  &= \left(\frac{1}{a_n}\sup_{y}\int_{a_n t}^\infty 
           \frac{\bar{F}_y(r)}{\bar{F}_y(a_n t)}\,dr\right) 
       \sum_{j=0}^\infty 
         \mathbb{E}_x\!\left[\mathbb{P}_x(\tau_{j+1} > a_n t \mid X_j, \tau_j)\,
         \mathds{1}_{\{\tau_j \le a_n t\}}\right] \\
  &= \left(\frac{1}{a_n}\sup_{y}\int_{a_n t}^\infty 
           \frac{\bar{F}_y(r)}{\bar{F}_y(a_n t)}\,dr\right) 
       \sum_{j=0}^\infty 
         \mathbb{E}_x\!\left[\mathbb{P}_x(\tau_{j} \le a_n t < \tau_{j+1})\right] \\
  &= \frac{1}{a_n}\sup_{y}\int_{a_n t}^\infty 
       \frac{\bar{F}_y(r)}{\bar{F}_y(a_n t)}\,dr.
\end{align*}
  Thus,
\[
  \lim_{t\to 0} \limsup_{n\to\infty} \sup_x d^n_x(t) 
  = \lim_{t\to 0} \limsup_{n\to\infty} 
    \frac{1}{a_n}\sup_{x}\int_{a_n t}^\infty 
      \frac{\bar{F}_x(r)}{\bar{F}_x(a_n t)}\,dr 
  = 0,
\]
  by condition \eqref{time_scale_main_cond}.
\end{proof}
\begin{remark}
As an example, consider the case when $F_x = F$ does not depend on $x$ (so that all holding times are identically distributed).  
In this situation, condition \eqref{time_scale_main_cond} holds if the tails of $F$ are exponential or polynomial.

Indeed, if
\[
  c_1 e^{-\alpha u} \le \bar{F}(u) \le c_2 e^{-\alpha u},
  \]
  for some constants $c_1, c_2 > 0$, and $\alpha > 0$ then
\[
  \frac{1}{a_n}\int_{a_n t}^\infty \frac{\bar{F}(r)}{\bar{F}(a_n t)}\,dr 
  \le {c_2\over c_1} \frac{1}{a_n}\int_{a_n t}^\infty e^{-\alpha r+\alpha a_n t}\,dr 
  = {c_2\over \alpha c_1} \frac{1}{a_n} \to 0, \quad n\to\infty.
\]

If instead for $c_1, c_2 > 0$ and $\alpha > 1$
\[
  c_1 u^{-\alpha} \le \bar{F}(u) \le c_2 u^{-\alpha},
\]
then
\[
  \frac{1}{a_n}\int_{a_n t}^\infty \frac{\bar{F}(r)}{\bar{F}(a_n t)}\,dr 
  \le \frac{c_2 (a_n t)^\alpha}{c_1 a_n }\int_{a_n t}^\infty r^{-\alpha} dr
  = {c_2\over (\alpha-1) c_1}t \to 0, \quad t\to 0.
\]
\end{remark}

\bibliographystyle{unsrt}  
\bibliography{references}

\end{document}